\documentclass[reqno, 11pt]{amsart}
\usepackage{url}
\usepackage[colorlinks, linkcolor=blue, citecolor=blue, urlcolor=blue]{hyperref}
\usepackage{times}

\newtheorem{theorem}{Theorem}[section]
\newtheorem{lemma}[theorem]{Lemma}

\newtheorem{corollary}[theorem]{Corollary}

\theoremstyle{definition}

\newtheorem{remark}[theorem]{Remark}

\numberwithin{equation}{section}

\usepackage{amsmath}
\usepackage{amsthm}
\usepackage{amsfonts}
\usepackage{amssymb}
\usepackage{mathrsfs}
\usepackage{amsxtra}
\usepackage{pifont}
\usepackage{amsbsy}

\usepackage{graphicx}

\newskip\aline \newskip\halfaline
\def\skipaline{\vskip\aline}

\def\qedbox{$\rlap{$\sqcap$}\sqcup$}
\def\qed{\nobreak\hfill\penalty250 \hbox{}\nobreak\hfill\qedbox\skipaline}

\DeclareMathOperator{\Aff}{\mathbf{Aff}}

\newcommand{\bC}{\mathbb{C}}

\newcommand{\bZ}{\mathbb{Z}}

\newcommand{\eR}{\mathscr{R}}
\newcommand{\eS}{\mathscr{S}}

\newcommand{\rp}{\,\big)}
\newcommand{\lp}{\big(\,}

\newcommand{\Rp}{\,\Big)}
\newcommand{\Lp}{\Big(\,}

\begin{document}

\title{On  the polynomial equation $P(Q(x_1,...,x_m))=Q(P(x_1),...,P(x_m))$}
\author{Arseny Mingajev}
\address{Department of Mathematics, Trinity University, San Antonio, TX78212}
\date{\today}

\begin{abstract} 
    We consider the equation $P(Q(x_1,...,x_\nu))=Q(P(x_1),...,P(x_\nu))$ in polynomials $Q\in\bC[x_1,\dotsc,x_\nu]$, $P\in\bC[x]$, and prove that if $\deg(P)>1$, then it is only solvable in polynomials that are affinely conjugate to monomials. 
\end{abstract}

\maketitle

\section{Introduction}
In 1922 and 1923, Gaston Julia \cite{Julia}  and Joseph Ritt \cite{Ritt}  proved that, if two polynomials in one variable $P(x)$ and $Q(x)$ commute with each other, then one of three statements is true:

\begin{enumerate}

\item  (Ritt, \cite{Ritt})$P$ and $Q$ are both affinely conjugate to monomials; 
\item  (Ritt, \cite{Ritt}) $P$ and $Q$ are both affinely conjugate to Chebyshev polynomials;

\item  (Julia, \cite{Julia} ) $P^{\circ\nu}=Q^{\circ\mu}$ for some positive integers $\nu,\mu.$

\end{enumerate}

An elementary proof of this theorem was published by Block and Thielman  in a 1951 paper,  \cite{BT}.    The results of  Ritt and Julia can be interpreted as results about commuting  rational maps $\mathbb{CP}^1\to\mathbb{CP}^1$. The commutativity of rational maps $\mathbb{CP}^N\to\mathbb{CP}^N$ is still a problem  that is being investigated; see  \cite{JKS} and the references therein.

In this paper, we consider another  multivariable version of this problem. For  each  natural number $\nu>1$  we denote by $\eS_\nu$   the subset of  $\bC[x]\times \bC[x_1,\dotsc, x_\nu]$ consisting of polynomials $P\in \bC[x]$ and $Q\in \bC[x_1,\dotsc, x_\nu]$, $\deg P,\deg Q, \geq 1$,  satisfying the functional equation
\begin{equation}\label{main}
P\big(\, Q(x_1,\dotsc,x_\nu)\rp=Q\lp P(x_1),\dotsc, P(x_\nu)\,\big)
\end{equation}
The main goal of this paper is to describe the structure of $\eS_\nu$ for $\nu\geq 2$. A few elementary observations are in order.

Denote by  $\Aff$  the group of affine bijections 
\[
\bC\ni z \mapsto az+b \in \bC,  \;\; a,b\in \bC, \;\;a\neq 0. 
\]
For any natural number $n$, the group $\Aff$ acts  on $\bC[z_1,\dotsc, z_n]$ by conjugation,
\[
\Aff\times \bC[z_1,\dotsc, z_\nu]\ni (\alpha , R)\mapsto C_\alpha (R)\in \bC[z_1,\dotsc, z_\nu],
\]
\[
\;C_\alpha (R)(z_1,\dotsc, z_\nu]= \alpha\circ R\lp \alpha^{-1}(z_1),\dotsc \alpha^{-1}(z_\nu)\rp.
\]
 An elementary computation shows that
 \[
 \lp P, Q\rp\in \eS_\nu\iff  \lp C_\alpha(P), C_\alpha(Q)\rp \in \eS_\nu,\;\;\forall \alpha\in \Aff.
 \]
 In other words, $\eS_\nu$ is invariant with respect to the action of $\Aff$.  
 
 \begin{lemma}\label{lemma: 11} Any polynomial $P(x)\in\mathbb{C}[x]$, $\deg P>0$ is affinely  conjugate to a polynomial vanishing at $0$. Moreover, if $\deg P>1$, then $P$ is affinely conjugate  to a polynomial with leading coefficient $1$ and vanishing at $0$.
  \end{lemma}

\begin{proof} Let  $\lambda\in \Aff$,  $\lambda (z)= z-r$, where $r$ is a root of the equation $p(x)-z=0$.  Then $\lambda^{-1}(0)=r$ and 
\[
C_\lambda (P)(0 )= P\lp \lambda^{-1}(0)\rp-r = P(r)-r=0.
\]
This proves the first part. Suppose now that $d=\deg P>1$ and $P(0)=0$.  If $\alpha(z)=cz$, $c\neq 0$, then $C_\alpha (P)(z)= cP(z/c)$. If $p_d$ is the leading coefficient of $P$, then the leading coefficient of $C_\alpha(P)$ is $c^{1-d}p_d$. Since $d>1$ we can find $c$ such that $c^{1-d}p_d=1$.
  \end{proof}

 For any multi-index $\alpha\in \bZ_{\geq 0}^\nu$ we set 
 \[
 |\alpha|=\alpha_1+\cdots+\alpha_\nu,\;\; x^\alpha:=x_1^{\alpha_1}\cdots x_\nu^{\alpha_\nu}.
 \]
 We can now state our  main result.

\begin{theorem}\label{th: main}  Let $P\in \bC[x]$, $Q\in \bC[x_1,\dotsc, x_\nu]$, $\nu\geq 2$, $\deg P=n>1$.  If $(P,Q)\in \eS_\nu$ and
\begin{equation}\label{nondeg}
\forall i\in \{1,\dotsc,\nu\},\;\;Q\not\in\bC[z_i],
\end{equation}
then there exists an affine transformation $\lambda \in \Aff$,  a constant $c\in\bC$,    and  a nonzero multi-index $\alpha \in\bZ_{\geq 0}^\nu$  such that 
\[
C_\lambda (P)(x)=x^n,\;\; C_\lambda(Q)(x_1,\dotsc, x_\nu)= cx^\alpha, \;\;c^{n-1}=1.
\]
\end{theorem}

\begin{remark}\label{rem: main} (a) Observe  that if $(P,Q)\in \eS_\nu$  are nontrivial monomials $Q=cx^\alpha$   and  $P(x)=x^n$, then the condition $c^{n-1}=1$ is automatically  satisfied.

\smallskip

\noindent (b)   The condition (\ref{nondeg}) is necessary. Indeed , if $Q$ depended only on one variable, then   according to \cite{Ritt}  the   equation (\ref{main}) has nonmonomial solutions.\qed\end{remark}

The proof will be by induction.

\section{The two-variable case}

Let 
$$
P(x)=\sum_{k=0}^na_kx^k.
$$
We can view $Q(x,y)$ either as a polynomial in $\bC[y][x]$ or as a polynomial in $\bC[x][y]$.
In the first case we  can express $Q$ as
\[
Q(x,y)=\sum_{i=0}^mQ_i(y)x^i.
\]
and in the second case  we can express  $Q$ as 
\[
Q(x,y)=\sum_{i=0}^\ell R_i(x)y^i.
\]
\begin{lemma} Under the conditions of the main theorem \ref{th: main}, if $(P,Q)\in \eS_2$, then the leading coefficients $Q_m(y)$ and $R_\ell(x)$ are non-constant.
\end{lemma}

\begin{proof}  We argue  by contradiction.  In this  case $Q(x,y)$ can be expressed as 
\[
Q(x,y)=f(x)+g(y)+ xyq(x,y),
\]
 where the highest power of $x$ in $q(x,y)$ is less than $\deg(f)-1$, and the highest power of $y$ in $q(x,y)$ is less than $\deg(g)-1$. Moreover $\min(\deg(f),\deg(g))>1$. 
 
 The commutation equation (\ref{main}) can be rewritten as 
 $$
 P(f(x)+g(y)+xyq(x,y))=f(P(x))+g(P(y))+P(x)P(y)q(P(x),P(y)).
 $$
By the statement of the theorem \ref{th: main}, either $g(y)\neq0$ or $q(x,y)\neq0$. Suppose first that $q\neq 0$. If $\deg(f)=N$, the term with the highest power of $x$ in $xyq(x,y)$ is of the form $q_1(y)x^{N-a},a>0$.

The highest power of $x$ with a non-constant coefficient in $P\lp f(x)+g(y)+xyq(x,y)\rp $ is $N(n-1)+N-a=Nn-a$, $n=\deg P$.

However, only  $P(x)P(y)q(P(x),P(y))\in \bC[y][x]$   has  coefficients  that are nonconstant as functions of $y$. The highest $p$ such that the coefficient $C_p[y]$ of $x^p$ is  non-constant coefficients $n(N-a)=nN-na$. Since $n>1$, this is impossible.

 If $q(x,y)=0$, then  (\ref{main}) reads
 $$
 P(f(x)+g(y))=f(P(x))+g(P(y)).
 $$ 
 This is not  solvable because $\deg P>1$. Indeed  $P(f(x)+g(y))$ contains terms divisible by $xy$, and $f(P(x))+g(P(y))$ does not. 

 An analogous proof works for powers of $y$.
\end{proof}

The equality $P(Q(x,y))=Q(P(x),P(y))$ can be rewritten as 
$$
\sum_{k=0}^na_k\Lp \sum_{i=0}^mQ_i(y)x^i\Rp^k=\sum_{i=0}^mQ_i(P(y))P(x)^i.
$$
Equating the terms with the highest power of $x$ we deduce 
\begin{equation}\label{1}
Q_m(y)^n=a_n^{m-1}Q_m(P(y)).
\end{equation}

Denote by $\eR$ the  set roots of $Q_m(y)$. For $r\in \eR$ we denote by $m(r)$ the multiplicity of $r$. Then 
$$
Q_m(y)=q_m\prod_{r\in\eR}(y-r)^{m(r)}, \;\;q_m\neq 0,
$$
and (\ref{1}) implies that
\begin{equation}\label{2}
 q_m^n\prod_{r\in\eR}(y-r)^{nm(r)}=a_n^{m-1}\underbrace{\prod_{r\in\eR}\lp P(y)-r\rp^{m(r)}}_{R(y)}.
\end{equation}
\begin{lemma} The polynomial  $Q_m$ has a single root of multiplicity  $m$, i.e., $|\eR|=1$.  Moreover, if $r_0$ is this unique root of $Q_m$, then $P(y)=(y-r_0)^n+r_0$. 
\end{lemma}

\begin{proof}   The equality (\ref{2}) shows that there exists a map $\phi:\eR\to\eR$ such that
\begin{equation}\label{3}
P(r)= \phi(r),\;\;\forall r\in \eR.
\end{equation}
On the other hand, for any $r\in\eR$, the equality (\ref{2}) implies that the nonempty set $P^{-1}(r)$ is contained in $\eR$. Indeed, if $P(y_0)=r$, then $R(y_0)=0$ so $Q_m(y_0)=0$ and thus $y_0\in \eR$. Hence  $\phi$ is surjective. Since $\eR$ is finite, $\phi$ is also bijective.

We claim  that for any $r\in \eR$ we have  $P^{-1}(r)=\{\phi^{-1}(r)\}$.  Set $r'=\phi^{-1}(r)$. The equality (\ref{3}) shows that  $\phi^{-1}(r)\in P^{-1}(r)$. Conversely,  for any $r'' \in P^{-1}(r)$ we have $P(r'')=r=\phi(r'')$ and since $\phi$ is bijective we deduce $r''=\phi^{-1}(r)=r'$.

We conclude that for any $r\in \eR$ we have
\begin{equation}\label{4}
P(y)=\lp y-\phi^{-1}(r)\rp^n+r.
\end{equation}
Let  $r_1,r_2\in \eR$. Set $s_j:=\phi^{-1}(r_j)$, $j=1,2$. Then
\[
(y-s_1)^n+r_1=P(y)=(y-s_2)^n+r_2,\;\;\forall y\in \bC.
\]
The Newton binomial formula implies that
\[
y^n-ns_1y^{n-1}+\cdots +(-1)^n s_1^n+r_1=  y^n-ns_2y^{n-1}+\cdots +(-1)^n s_2^n+r_2.
\]
We have $n-1\geq 1$. Equating the coefficients of  $y^{n-1}$  in the above equality we deduce $s_1=s_2$.   Since $\phi$ is bijective we conclude that $r_1=r_2$  so that $\eR$ consists of a single root $r_0$. The equality (\ref{4}) shows that    $P(y)=(y-r_0)^n+r_0$.\end{proof}

 The above discussion shows that  if $(P, Q)\in \eS_2$, and $\deg P>1$ then $P$ is  affinely conjugate to  the monomial $x^n$. The equation (\ref{main}) becomes
 \begin{equation}\label{main1}
 Q( x^n, y^n)= Q(x,y)^n. 
 \end{equation}
 To tackle the two-variable case, we must first learn how to solve the much simpler one-variable case
$$
Q(x)^n= c^n\prod_{k=1}^p(x-x_k)^n=Q( x^n)=c\prod_{k=1}^p( x^n-x_k).
$$ 
For $n>1$, the polynomials 
$$
\prod_{k=1}^p(x-x_k)^n, \prod_{k=1}^p( x^n-x_k)
$$  
must have the same roots. Each non-zero root on the right increases the difference between the number of roots on the right and left by $n-1$. Since $n>1$, all the roots must be zero, and $Q(x)=cx^p$, where $c$ is some constant.

\begin{corollary}\label{cor: 23} If $H(x,y)$ is homogeneous of degree $p$, then $H(x,y)=y^pH(\nu,1)$ where $z=\frac{x}{y}$. So if $ H(x,y)^n=H( x^n, y^n)$, then $H(z,1)=cz^q$, and $H(x,y)=cx^qy^{p-q}$, $c^{n-1}=1$. \qed
\end{corollary}

We can now complete  the proof of Theorem \ref{th: main} in the two-variable case, $\nu=2$. We write
\[
Q(x,y)= \sum_{k=0}^m H_k(x,y),\;\;m\geq 1,
\]
 where $H_k(x,y)$ is a homogeneous polynomials of degree  $k$. 
\begin{lemma}\label{lemma: 24}
    If $H_0\neq0$, then the only solution to $(\ref{main1})$ is $Q(x,y)\equiv H_0.$ 
\end{lemma}

\begin{proof}
    The sum $\left(\sum_{k=0}^m H_k\right)^n$ can be written as $\sum_{k=0}^{mn}\overline{H_k}$, where $\overline{H_k}$ are homogeneous polynomials of degree $k$ that are sums of products of $n$ $H_i$-s, where $0\leq i\leq k.$ Then (\ref{main1}) becomes 
    \begin{equation}\label{7}
        \sum_{k=0}^{mn}\overline{H_k(x,y)}=\sum_{k=0}^mH_k(x^n,y^n). 
    \end{equation}
    
    Since $H_k(x^n,y^n)$ are homogeneous polynomials of degree $nk$, it must be true that $\overline{H_{nj}(x,y)}=H_j(x^n,y^n)$ and $\overline{H_k}(x,y)=0$  if $k$ is not a multiple of $n$. 
    
  We have $\overline{H_1}=nH_0^{n-1}H_1=0$, therefore $H_1=0$. Notice that the right hand side of (\ref{7}) changes as well, since now the homogeneous polynomial of the smallest positive degree is $H_2(x^n,y^n)$. Therefore, $\overline{H_k}\equiv0$ for $1< k<2n$. In general, if $H_k\equiv0$ for all $0<k\leq p$, then $\overline{H_k}\equiv0$ for $0<k<(n+1)p$. Since $n>1, p>1$ and $1<p+1<2p$, $\overline{H_p}\equiv0$ implies $\overline{H_{p+1}}\equiv 0$. Suppose that we have proven that $H_k\equiv 0$ for $k=p$. To prove it for $k=p+1$, notice $\overline{H_{p+1}}\equiv0$, and the expression for $\overline{H_{p+1}}$ involves only $H_k$ or degree $\leq p+1$. However, all $H_k, 1\leq k\leq p$ are null, so $\overline{H_{p+1}}=nH_0^{n-1}H_{p+1}\equiv 0$, therefore $H_{p+1}\equiv 0$.

\end{proof}

\begin{lemma}\label{lemma: 25}
    $Q(x,y)=H_m(x,y)$.
\end{lemma}
\begin{proof}
    Since $H_0=0$, $\overline{H_k}\equiv0$ for $k<n$ and $\overline{H_n}=H_1^n(x,y)=H_1(x^n,y^n)$. Denoting $H_k$ as $H'_{k-1}$ we get that unless $H'_0=H_1\equiv0$, every term after it must vanish by the exact same proof as in Lemma \ref{lemma: 24}, but with $H$-s replaced with $H'$-s. Therefore, $H_1\equiv0$. Continuing by induction, we obtain that every term in $\sum_kH_k$ other than the highest order one must vanish. 
\end{proof}

Since $Q(x,y)=H_m(x,y)$ is a homogeneous polynomial, it must be conjugate to a monomial to satisfy (\ref{main1}) by Corollary \ref{cor: 23}.

\section{The inductive step}

We  first  solve (\ref{main1}) in an arbitrary number of variables.

\begin{lemma}\label{lemma: 3.1}  Given that $Q(x_1,...,x_\nu )$ is a polynomial in $\nu$, $\deg Q\geq 1$ variables and $n$ is a natural number greater than 1, the equation $Q( x_1^n,..., x_\nu^n)= Q(x_1,...,x_\nu)^n$ can only be satisfied if $Q$ is a monomial.
\end{lemma}

\begin{proof}The case where $\nu=2$  was  discussed previously. Suppose that this has been proven for polynomials  $Q$ in $\nu$ variables.  To prove it for polynomials in $k=\nu+1$ variables, we  employ the homogeneous decomposition from the previous section. Since the number of variables does not affect the arguments in the proofs of Lemmas \ref{lemma: 24} and \ref{lemma: 25}, we deduce that $Q$  must be a homogeneous polynomial of degree $p$, $Q=H_{p}$, where $H_p$ is a homogeneous polynomial of degree $p>0$. Then, the equation from the statement takes the form 
$$
H_p( x_1^n,...,x_{m+1}^n)= H_p(x_1,...,x_{\nu+1})^n.
$$ 
Since $H_p$ is homogeneous of degree $p$, 
$$
H_p(x_1,...,x_{m+1})=x_{\nu+1}^pH_p(y_1,...,y_\nu,1), y_i=\frac{x_i}{x_{\nu+1}}.
$$
Since $H_p(y_1,...,y_\nu, 1)$ is a polynomial in $\nu$ variables, it must be a monomial by the inductive hypothesis. Thus, $Q$ must also be a monomial.\end{proof}

In other words  if $(P,Q)\in \eS_{\nu+1}$, $Q\neq 0$ and $P$ is a monomial, the $Q$ must also be a monomial. In view of Remark \ref{rem: main}  to prove  Theorem \ref{th: main} it suffices  to show that if $(P,Q)\in \eS_\nu$ then $P$ is a monomial.

\begin{lemma}\label{lemma: 3.2}
    Under the conditions of Theorem \ref{th: main}, for all $i\in\{1,\dotsc,\nu\}$ the highest power of $x_i$ has a non-constant coefficient. 
\end{lemma}

\begin{proof}
    Express $Q(x_1,...,x_\nu)$ as $f(x_\nu)+g(x_1,...,x_{\nu-1})+x_\nu q(x_1,...,x_\nu)$, where $\deg(f)=N$ and the maximum power of $x_\nu$ in $x_\nu q(x_1,...,x_\nu)$ is equal to $N-a, a>0$. The commutation relation can be written as $$P(f(x_\nu)+g(x_1,...,x_{\nu-1})+x_\nu q(x_1,...,x_\nu))=$$$$=f(P(x_\nu))+g(P(x_1),...,P(x_{\nu-1}))+P(x_\nu)q(P(x_1),...,P(x_\nu)).$$ By the conditions of the main theorem \ref{th: main}, either $g\neq0$ or $q\neq 0$. If $q\neq0$, this equation is not solvable, since the maximum power of $x_\nu$ with non-constant coefficients on the left hand side is $N(\deg(P)-1)+N-a=N\deg(P)-a$, while on the right hand side it is just $(N-a)\deg(P)<N\deg(P)-a$. If $q=0$, the equation is still not solvable since $\deg(P)>1$, so $P(f(x_\nu)+g(x_1,...,x_{\nu-1}))$ contains terms divisible by $x_\nu x_i$, $i\in\{1,\dotsc,\nu-1\}$, while $f(P(x_\nu))+g(P(x_1),\dotsc,P(x_{\nu-1}))$ does not. The exact same argument holds for the rest of $x_i$-s.  
\end{proof}

Let

$$
P(x)=\sum_{k=0}^na_kx^k.
$$

Let
\[
Q(x_1,\dotsc,x_\nu)=\sum_{i=0}^mQ_i(x_1,\dotsc,x_{\nu-1})x_\nu^i.
\]

By Lemma $\ref{lemma: 11}$, $P$ can be affinely conjugated to a polynomial that vanishes at $0$ and has leading coefficient $1$. The commutation relation can be written as 

\[
\sum_{k=1}^na_k(\sum_{i=0}^mQ_i(x_1,\dotsc,x_{\nu-1})x_\nu^i)^k=\sum_{i=0}^mQ_i(P(x_1),\dotsc,P(x_{\nu-1}))P(x_\nu)^i
\]

Equating the maximal powers of $x_\nu$ on both sides, we get that 

\begin{equation}\label{3.1}
    Q_m(x_1,\dotsc,x_{\nu-1})^n=Q_m(P(x_1,\dotsc,x_{\nu-1}))
\end{equation}

Under the conditions of the main theorem \ref{th: main}, $Q_m$ is non-constant by Lemma $\ref{lemma: 3.2}$. 

\begin{lemma}\label{lemma: 3.3}
    Given that $Q(x_1,...,x_k)\in\mathbb{C}[x_1,...,x_k]$ is non-constant, and $P(y)\in\mathbb{C}[y], P(0)=0, \deg(P)=n>1$, $(\ref{3.1})$ is only solvable in polynomials that are affinely conjugate to monomials. 
\end{lemma} 

\begin{proof}
    This has been proven for $k=1$ in section 2. Suppose that the theorem is proven for $k=m-1$. Suppose that $Q_0(x_1,...,x_{m-1})\equiv Q(x_1,...,x_{m-1},0)$ is non-constant. Then $$Q_0(P(x_1),...,P(x_{m-1}))=Q(P(x_1),...,P(x_{m-1}),0)=$$$$=Q(P(x_1),...,P(x_{m-1}),P(0))=Q(x_1,...,x_{m-1},0)^n=Q_0(x_1,...,x_{m-1})^n.$$
    Unless $Q_0$ is constant, $P$ must be conjugate to a monomial by the inductive hypothesis. By Lemma \ref{lemma: 3.1}, $Q$ must be conjugate to a monomial as well. 

    Suppose now that $Q_0\equiv c$, where $c$ is some non-zero constant. Then $$Q=c+x_m^kq(x_1,...,x_m), k>0, q(x_1,...,x_{m-1},0)\neq0.$$  The equation in the statement of the corollary becomes $$c+P(x_m)^kq(P(x_1),...,P(x_m))=(c+x_m^kq(x_1,...,x_m))^n.$$ Let the smallest power of $x_m$ in $P(x_m)$ be $t>0$. If $c\neq0$, then the smallest positive power of $x_m$ on the right hand side is $k$, while the smallest positive power of $x_m$ on the left hand side is $tk$. Since $t>0$, for these powers to be equal to each other $t=1$, making $P(x)=ax+x^2p(x)$, and $Q(x_1,...,x_m)=c+x_m^{k}q_1(x_1,...,x_{m-1})+x_m^{k+1}q_2(x_1,...,x_m), q_1\neq0.$
    Then 
    \begin{equation}\label{3.2}
        \begin{split}
            c+P(x_m)^{k}q_1(P(x_1),...,P(x_{m-1}))+P(x_m)^{k+1}q_2(P(x_1),...,P(x_m))= \\
            (c+x_m^{k}q_1(x_1,...,x_{m-1})+x_m^{k+1}q_2(x_1,...,x_m))^n.
        \end{split}
    \end{equation}
    Setting the terms with lowest positive power of $x_m$ equal each other, $$a^kx_m^{k}q_1(P(x_1),...,P(x_{m-1}))=nc^{n-1}x_m^{k}q_1(x_1,...,x_{m-1}).$$ 
    However, since $\deg P>1$, this is only possible when $q_1=b\neq0$ is a constant.
    Setting $x_1=\dotsc=x_{m-1}=0$ and $q_0\equiv q_2(0,...,0,x_m)$, $(\ref{3.2})$ simplifies to 
    \begin{equation}\label{3.3}
        c+bP(x_m)^{k}+P(x_m)^{k+1}q_0(P(x_m))=(c+bx_m^{k}+x_m^{k+1}q_0(x_m))^n.
    \end{equation}
    Since $(\ref{3.3})$ is of the type $Q(P)=Q^n$ in one variable, it is only solvable when $P(x)$ is conjugate to a monomial by section 2. Then by Lemma $\ref{lemma: 3.1}$, this means that $Q$ must also be conjugate to a monomial. 

    If $c=0$, then $$P(x_m)^kq(P(x_1),...,P(x_m))=x_m^{kn}q(x_1,...,x_m)^n,$$ and since $q(P(x_1),...,P(x_m))$ is not divisible by $x_m,$ then $P(x_m)^k$ has to divide $x_m^{kn}$. However, since $\deg(P)=n$, this only leaves the possibility that $P=x^n$, and once again $Q$ has to be a monomial by the previous lemma.
\end{proof}

That proves the main Theorem $\ref{th: main}$.

\end{document}